\documentclass[a4paper,11pt]{article}
\usepackage[noargminspace]{okicmd}
\usepackage{okithm}
\usepackage{amsthm,thmtools}
\usepackage[top=80pt, bottom=90pt, left=70pt, right=70pt]{geometry}
\usepackage{comment}
\usepackage{color}
\usepackage[colorlinks=true,linkcolor=blue,citecolor=blue]{hyperref}
\usepackage{cleveref}
\usepackage{cases}
\usepackage{enumitem}
\usepackage{xspace}

\usepackage{tikz}
\usetikzlibrary{positioning}
\usetikzlibrary{intersections}
\usetikzlibrary{calc}
\usetikzlibrary{arrows.meta}
\usetikzlibrary{shapes.geometric}
\usetikzlibrary{shapes.misc}
\usetikzlibrary {decorations.pathmorphing}

\newcounter{nodecount}

\declaretheorem[style=plain,name=Theorem]{theorem}
\declaretheorem[style=plain,sibling=theorem,name=Lemma]{lemma}
\declaretheorem[style=plain,sibling=theorem,name=Proposition]{proposition}

\crefname{theorem}{Theorem}{Theorems}
\crefname{proposition}{Proposition}{Propositions}
\crefname{lemma}{Lemma}{Lemmas}
\crefname{exmp}{Example}{Examples}
\crefname{corollary}{Corollary}{Corollarys}
\crefname{claim}{Claim}{Claims}
\crefname{remark}{Remark}{Remarks}
\crefname{section}{Section}{Sections}

\newcommand{\BP}{\mathrm{P}}

\newcommand{\exaxiom}{{\rm ($\Delta$-EXC)}\xspace}
\newcommand{\exaxiomm}{{\rm (BS-EXC)}\xspace}
\newcommand{\LP}{{\rm (LP)}\xspace}
\newcommand{\ub}{r}
\newcommand{\lb}{q}

\newcommand{\dep}{\mathrm{dep}}

\DeclareMathOperator{\conv}{conv}
\DeclareMathOperator{\cone}{cone}

\makeatletter

\@addtoreset{equation}{section}
\makeatother

\usepackage{bm}
\allowdisplaybreaks

\begin{document}
	\title{Characterizations of the set of integer points\\in an integral bisubmodular polyhedron}
	\author{Yuni Iwamasa\thanks{Department of Communications and Computer Engineering, Graduate School of Informatics, Kyoto University, Kyoto 606-8501, Japan.
			Email: \texttt{iwamasa@i.kyoto-u.ac.jp}}}
	\date{\today}
	\maketitle
	
\begin{abstract}
In this note, we provide two characterizations of the set of integer points in an integral bisubmodular polyhedron.
Our characterizations do not require the assumption that a given set satisfies the hole-freeness, i.e.,
the set of integer points in its convex hull coincides with the original set.
One is a natural multiset generalization of the exchange axiom of a delta-matroid,
and the other comes from the notion of the tangent cone of an integral bisubmodular polyhedron.
\end{abstract}
\begin{quote}
\end{quote}

\section{Introduction}\label{sec:intro}

The set of integer points in an integral submodular polyhedron
is called an \emph{M-convex set}~\cite{AM/M96}.
This is a polyhedral generalization of a matroid defined by the base family.
Indeed,
a matroid coincides with the set of integer points in the base polytope of the matroid,
where the base polytope is the integral submodular polytope with respect to the rank function of the matroid.
Hence, the concept of M-convex set can be obtained by generalizing ``the rank function of the matroid'' to ``a general integral submodular function''
in the above argument on the matroid base family.

The relation between integral submodular functions and M-convex sets
can be viewed as a discrete version of the conjugacy.
In convex analysis,
it is well-known~\cite{book/Rockafellar} that
there is a one-to-one correspondence between positively homogeneous closed proper convex functions
and (the indicator functions of) closed convex sets via the Legendre--Fenchel transformation.
From this viewpoint,
(the Lov\'{a}sz extension of) an integral submodular function corresponds to a positively homogeneous closed proper convex function,
and an M-convex set is its corresponding convex set.
By combining convex analysis and submodular/matroid theory,
Murota developed a theory of ``discrete convexity'' on integer lattices, called \emph{discrete convex analysis (DCA)}~\cite{book/Murota03}.
In DCA,
the M-convexity is a fundamental convex concept of a set of integer points in the integer lattice.
We here note that an M-convex set is clearly \emph{hole-free}, i.e.,
the intersection of the convex hull of an M-convex set and the integer lattice coincides with the original set,
which is a naturally required condition for a set of integer points to be ``convex.''

An M-convex set admits several exchange-type axioms,
which are multiset versions of the exchange axioms of a matroid.
Such axioms are more useful in checking the M-convexity of a given set than the polyhedral definition of an M-convex set,
and
are used in devising a simple greedy algorithm for a linear optimization problem over an M-convex set.
It is worth mentioning that the fact that ``a set of integer points satisfying the exchange axioms is hole-free'' is nontrivial.

In this note, we address a bisubmodular/delta-matroidal generalization of M-convex sets.
The set of integer points in an integral bisubmodular polyhedron
is called a \emph{BS-convex set}~\cite{DO/F14},
which can also be viewed as a polyhedral generalization of a \emph{delta-matroid}~\cite{AM/DH86,MP/B87,DM/CK88,MP/Q88};
a delta-matroid coincides with
the set of integer points in its feasible set polytope of the delta-matroid ($=$ the integral bisubmodular polytope with respect to the rank function of the delta-matroid).
A BS-convex set can also be considered as a discrete convex set,
since it corresponds to the conjugate of (the Lov\'{a}sz extension of) an integral bisubmodular function
that can be viewed as positively homogeneous proper closed convex.

The present article particularly focuses on exchange-type characterizations of BS-convex sets.
A few exchange-type characterizations of BS-convex sets are known~\cite{JORSJ/AFN94,SIDMA/BC95},
but all of them additionally require that a given set is hole-free.
For example,
a \emph{jump system}, introduced by Bouchet and Cunningham~\cite{SIDMA/BC95},
is defined by a certain exchange axiom,
and it is known that a hole-free set is BS-convex if and only if it is a jump system.
However, the hole-freeness assumption of a given set cannot be omitted,
since there exists a jump system that is not hole-free.
So far, no exchange-type characterization of a BS-convex set without any assumption on a given set
has been known.

Our main result is to introduce two exchange-type characterizations of a BS-convex set not requiring any assumption on the set.
The one is a natural multiset version of the exchange axiom of a delta-matroid.
The other one is based on the notion of the tangent cone of an integral bisubmodular polyhedron.
Both can be more useful in checking the BS-convexity of a given set than the polyhedral definition of a BS-convex set.

\section{Preliminaries}
Let $\R$ and $\Z$ denote the sets of reals and integers, respectively.
For a positive integer $k$, let $[k] \defeq \set{1,2,\dots,k}$.
For an index set $I$ and a tuple $t \in \R^I$ (or $t \in \Z^I$ or $t \in \set{-1,0,1}^I$),
its \emph{support} $\supp(t)$ is defined as the set of indices $i \in I$ such that the $i$-th element $t(i)$ of $t$ is nonzero,
that is,
$\supp(t) \defeq \set{i \in I}[t(i) \neq 0]$.

Let $V$ be a nonempty finite set.
For $p, x \in \R^V$,
we define $\inpr{p}{x} \defeq \sum_{u \in V} p(u) x(u)$.
For $X \subseteq \R^V$,
we denote by $\conv(X)$ and $\cone(X)$
the convex hull and conical hull of $X$,
respectively. That is,
\begin{align*}
    \conv(X) &\defeq \set*{ \sum_{i = 1}^n \lambda_i x_i }[$n:$ positive integer, $x_1,\dots,x_n \in X$, $\lambda_1,\dots, \lambda_n \geq 0$, $\sum_{i = 1}^n \lambda_i = 1$],\\
    \cone(X) &\defeq \set*{ \sum_{i = 1}^n \mu_i x_i }[$n:$ positive integer, $x_1,\dots,x_n \in X$, $\mu_1,\dots, \mu_n \geq 0$].
\end{align*}
A set $B \subseteq \Z^V$ of integer points
is said to be \emph{hole-free}
if $B$ coincides with the set of integer points in the convex hull of $B$,
i.e.,
$B = \conv(B) \cap \Z^V$.
For $u \in V$,
let $\chi_u$ denote the $u$-th unit vector.
For $p \in \Z^V$,
we define $\norm{p}_1$ as the $1$-norm $\sum_{u \in V} \abs{p(u)}$ of $p$.
We define $\Phi$ as the set of vectors $\alpha$ in $\set{-1,0,1}^V$ such that its $1$-norm $\norm{\alpha}_1$ is one or two,
namely,
\begin{align*}
    \Phi \defeq \set{\pm\chi_u}[ u \in V] \cup \set{\pm \chi_u \pm \chi_v}[u,v \in V,\ u \neq v].
\end{align*}
We also define
\begin{align*}
    \Phi(p, q) &\defeq \set{\alpha \in \Phi}[\norm{q - (p + \alpha)}_1 = \norm{q - p}_1 - \norm{\alpha}_1] \qquad (p, q \in \Z^V),\\
    \Phi_B(p, q) &\defeq \set{\alpha \in \Phi(p, q)}[p + \alpha \in B] \qquad (B \subseteq \Z^V, \ p,q \in B)
\end{align*}

We assume that any function $f$ on $\{-1, 0, 1\}^V$ appearing in the following satisfies $f(\zeros) = 0$,
where $\zeros$ denotes the all-zero vector.
A function $\funcdoms{f}{\set{-1, 0, 1}^V}{\R \cup \set{+\infty}}$ is said to be \emph{bisubmodular} (see e.g., \cite[Section~3.5~(b)]{book/Fujishige05})
if it satisfies the following bisubmodular inequalities;
\begin{align*}
    f(x) + f(y) \geq f(x \sqcap y) + f(x \sqcup y) \qquad (x, y \in \{-1, 0, 1\}^V),
\end{align*}
where the binary operations $\sqcap$ and $\sqcup$ are defined by
\begin{align*}
    (x \sqcap y)(u) \defeq
    \begin{cases}
    x(u) & \text{if $x(u) = y(u)$},\\
    0 & \text{if $x(u) \neq y(u)$},
    \end{cases}
    \qquad
    (x \sqcup y)(u) \defeq
    \begin{cases}
    x(u) & \text{if $x(u) = y(u)$ or $y(u) = 0$},\\
    y(u) & \text{if $x(u) = 0$},\\
    0 & \text{if $0 \neq x(u) \neq y(u) \neq 0$}
    \end{cases}
\end{align*}
for each $u \in V$.
It is known~\cite{MP/Q88} that
a function $\funcdoms{f}{\set{-1, 0, 1}^V}{\R \cup \set{+\infty}}$ is bisubmodular
if and only if its \emph{Lov\'{a}sz extension} $\funcdoms{\tilde{f}}{\R^V}{\R \cup \set{+\infty}}$
is (closed) convex;
the definition of Lov\'{a}sz extension of a function on $\set{-1, 0, 1}^V$ is given in \cite{MP/Q88}.
Since the Lov\'{a}sz extension of any function on $\set{-1, 0, 1}^V$ is positively homogeneous,
a bisubmodular function may play a role as a positively homogeneous proper closed convex function in this sense,
as described in Introduction.
We say that $f$ is \emph{integral}
if all finite values of $f$ are integral,
i.e., $f$ is a function from $\set{-1, 0, 1}^V$ to $\Z \cup \set{+\infty}$.

For a bisubmodular function $\funcdoms{f}{\set{-1, 0, 1}^V}{\R \cup \set{+\infty}}$,
its \emph{bisubmodular polyhedron} $\BP(f)$ is the polyhedron defined by
\begin{align*}
\BP(f) \defeq \set{ p \in \Z^V }[\inpr{p}{x} \leq f(x) \ (\forall x \in \set{-1,0,1}^V)].
\end{align*}
Note that the conjugate function $p \mapsto \sup_{x \in \R^V} \{\inpr{p}{x} - \tilde{f}(x)\}$
of the Lov\'{a}sz extension $\tilde{f}$ of $f$
is the indicator function of the bisubmodular polyhedron $\BP(f)$,
where $p \in \R^V$.
The tangent cone of $\BP(f)$ admits a simple representation as follows,
where for a convex set $C \subseteq \R^V$ and a point $p \in C$,
the \emph{tangent cone} of $C$ at $p$ is defined as $\cone(\set{ q - p }[q \in C])$.
\begin{lemma}[{\cite[Corollary~3.6]{MP/AF96}}]\label{lem:cone}
Let $\funcdoms{f}{\set{-1, 0, 1}^V}{\R \cup \set{+\infty}}$ be a bisubmodular function.
    For $p \in \BP(f)$,
    the tangent cone of $\BP(f)$ at $p$ coincides with $\cone(\set{\alpha \in \Phi}[ p + \epsilon \alpha \in \BP(f) \text{ for some } \epsilon > 0 ])$.
\end{lemma}
If $f$ is integral, then $\BP(f)$ is an integral polyhedron~\cite[Corollary~5.4]{SIDMA/BC95},
which is called an \emph{integral bisubmodular polyhedron}.

A nonempty set $B \subseteq \Z^V$ is called a \emph{BS-convex set}~\cite{DO/F14}
if there exists an integral bisubmodular function $\funcdoms{f}{\set{-1, 0, 1}^V}{\Z \cup \set{+\infty}}$
such that $B = \BP(f) \cap \Z^V$.
Note that the discrete conjugate function $p \mapsto \sup_{x \in \Z^V} \{\inpr{p}{x} - \tilde{f}(x)\}$
of the Lov\'{a}sz extension $\tilde{f}$
is the indicator function of the BS-convex set,
where $p \in \Z^V$.
In this sense,
a BS-convex set is ``discrete convex'' as described in Introduction.

A nonempty set $J \subseteq \Z^V$ is called a \emph{jump system}~\cite{SIDMA/BC95}
if it satisfies the following exchange axiom, called the \emph{2-step axiom}:
\begin{description}
\item[{\rm (J-EXC)}]
For any $p, q \in J$ and $u \in \supp(q- p)$,
\begin{itemize}
    \item there exists $\alpha \in \Phi_B(p, q)$ such that $u \in \supp(\alpha)$, or
    \item $\norm{p(u) - q(u)}_1 \geq 2$ and $p + 2 s_u \in J$ for the unique $s_u \in \set{\pm \chi_u}$ satisfying $s_u \in \Phi(p, q)$.
\end{itemize}
\end{description}
The above condition can be rephrased as:
For any $p, q \in J$ and $u \in \supp(q- p)$,
we have $p + s_u \in J$ for the unique $s_u \in \set{\pm \chi_u}$ satisfying $s_u \in \Phi(p, q)$,
or $p + s_u + s_v \in J$ for some $v \in V$ and $s_v \in \set{\pm \chi_v}$ satisfying $s_v \in \Phi(p + s_u, q)$.
This is the reason why (J-EXC) is called as the 2-step axiom.

A jump system is a generalization of a BS-convex set in the following sense.
\begin{proposition}[{\cite[Theorems~4.4 and~5.3]{SIDMA/BC95}}]\label{prop:jump system}
For any jump system $J \subseteq \Z^V$,
its convex hull $\conv(J)$ is an integral bisubmodular polyhedron.
Conversely,
for any integral bisubmodular function $\funcdoms{f}{\set{-1, 0, 1}^V}{\Z \cup \set{+\infty}}$,
the corresponding BS-convex set $\BP(f) \cap \Z^V$ is a jump system.
\end{proposition}

\Cref{prop:jump system} says that, for a hole-free set $B \subseteq \Z^V$, it is BS-convex
if and only if it is a jump system.
While all BS-convex sets are hole-free, which immediately follows from the definition,
a jump system is not hole-free in general.
For example, the set $\{0, 2\} \subseteq \Z$ is a jump system but is not hole-free.
Hence, the assumption of the hole-freeness of a given set cannot be omitted.

\section{Characterizations of a BS-convex set}
In this section, we present two exchange-type characterizations of a BS-convex set with no assumption on a given set.

The first one is the condition obtained from (J-EXC) by removing the second condition on the case of $\norm{p(u) - q(u)}_1 \geq 2$.
\begin{description}
\item[\exaxiom]
For any $p, q \in B$ and $u \in \supp(q-p)$,
there exists $\alpha \in \Phi_B(p, q)$ such that $u \in \supp(\alpha)$.
\end{description}
We here note that a nonempty set $B \subseteq \set{0,1}^V$ is called a \emph{delta-matroid}~\cite{AM/DH86,MP/B87,DM/CK88,MP/Q88}
if $B$ satisfies \exaxiom.
Hence
the condition \exaxiom is a natural multiset version of the exchange axiom of a delta-matroid.

The second one comes from the notion of the tangent cone.
By the convexity of $\BP(f)$,
for any $p, q \in \BP(f) \cap \Z^V$
the difference $q-p$ of $q$ and $p$ belongs to the tangent cone of $\BP(f)$ at $p$,
namely, $q-p \in \cone(\set{\alpha \in \Phi}[ p + \epsilon \alpha \in \BP(f) \text{ for some } \epsilon > 0 ])$ by \Cref{lem:cone}.
The following condition requires that
$q-p$ particularly belongs to $\cone(\Phi_B(p, q))$
and we can take the nonnegative combination coefficients of $q-p$ to be half-integral.
\begin{description}
\item[\exaxiomm]
For any $p, q \in B$, there exist $\alpha_1, \alpha_2, \dots, \alpha_k \in \Phi_B(p, q)$ such that $p + \sum_{i = 1}^k \alpha_i/2 = q$.
\end{description}

Our main result is the following:
\begin{theorem}\label{thm:characterization}
Let $B \subseteq \Z^V$ be a nonempty set of integer points.
The following conditions {\rm (a)}, {\rm (b)}, and {\rm (c)} are equivalent.
\begin{description}
    \item[{\rm (a)}] $B$ is a BS-convex set, i.e., $B = \BP(f) \cap \Z^V$ for some integral bisubmodular function $f$ on $\{-1, 0, 1\}^V$.
    \item[{\rm (b)}] $B$ satisfies \exaxiom.
    \item[{\rm (c)}] $B$ satisfies \exaxiomm.
\end{description}
\end{theorem}

The implication $\text{(c)} \Rightarrow \text{(b)}$ easily follows.
Indeed,
for any $p, q \in B$,
there exist $\alpha_1, \alpha_2, \dots, \alpha_k \in \Phi_B(p, q)$ such that $p + \sum_{i = 1}^k \alpha_i/2 = q$
by \exaxiomm.
Then, for each $u \in \supp(q-p)$ at least one of $\alpha_1, \alpha_2, \dots, \alpha_k$ must satisfy $u \in \supp(\alpha_i)$,
which implies \exaxiom.

The proofs of $\text{(b)} \Rightarrow \text{(a)}$ and $\text{(a)} \Rightarrow \text{(c)}$ are given in \Cref{subsec:b=>a,subsec:a=>c},
respectively.

\subsection{Proof of $\text{(b)} \Rightarrow \text{(a)}$}\label{subsec:b=>a}
Suppose that $B$ satisfies the condition \exaxiom.
Since \exaxiom is stronger than (J-EXC),
the set $B$ is a jump system.
Thus, by \Cref{prop:jump system},
we have $\conv(B) = \BP(f)$ for some integral bisubmodular function $f$ on $\{-1, 0, 1\}^V$.
Hence, it suffices to show that $B = \conv(B) \cap \Z^V$, i.e., $B$ is hole-free.

It is clear that $B \subseteq \conv(B) \cap \Z^V$.
Take any $p^* \in \conv(B) \cap \Z^V$.
In the case where $B$ is infinite,
we replace $B$ with the intersection $B \cap [-n\ones, n\ones]$ for sufficiently large $n \in \Z_+$
such that $p^* \in \conv(B \cap [-n\ones, n\ones])$,
where $\ones$ denotes the all-one vector in $\Z^V$.
Then the resulting $B$ is finite and satisfies \exaxiom.
Therefore, we can assume that $B$ is finite.

Let
\begin{align*}
    C := \set{(\lambda_p)_{p \in B} \in \R^B}[p^* = \sum_{p \in B} \lambda_p p,\ \sum_{p \in B} \lambda_p = 1,\ \lambda_p \geq 0 \ (p \in B)].
\end{align*}
Since $p^* \in \conv(B) \cap \Z^V$,
the set $C$ is nonempty.
Moreover $C$ is compact by the definition of $C$ and the finiteness of $B$.
Define a function $\funcdoms{\theta}{C}{\R}$ by
\begin{align*}
    \theta(\lambda) \defeq \sum_{p \in B} \lambda_p \norm{p^* - p}_1
\end{align*}
for $\lambda = (\lambda_p)_{p \in B} \in C$.
Since $\theta$ is continuous and $C$ is a compact nonempty set,
the function $\theta$ attains its infimum, i.e., there exists $\lambda^* \in C$ such that $\theta(\lambda^*) = \inf \theta (= \min \theta)$.
We can see that $p^* \in B$ if and only if $\min \theta = 0$.

Suppose, to the contrary, that $p^* \notin B$, i.e., $\min \theta > 0$.
Let $\lambda^* = (\lambda^*_p)_{p \in B} \in \argmin \theta$.
Note that $p^* \notin \supp(\lambda^*)$.
Take $\lb, \ub \in \supp(\lambda^*)$ such that
$\lb(u') < p^*(u') < \ub(u')$
for some $u' \in V$;
such $\lb$ and $\ub$ exist
since $\sum_{p \in \supp(\lambda^*)} \lambda^*_p p = p^*$ and $p^* \notin \supp(\lambda^*)$.
Here the following holds.
\begin{lemma}\label{lem:exchange}
There exist $\alpha_1, \alpha_2, \dots, \alpha_k \in \Phi_B(\lb, \ub)$ and $\beta_1, \beta_2, \dots, \beta_l \in \Phi_B(\ub, \lb)$
such that $\sum_{i = 1}^k \alpha_i + \sum_{j = 1}^l \beta_j = 0$.
\end{lemma}
\begin{proof}
Let $G$ denote an undirected graph (that can have self-loops) whose vertex set is $\supp(\ub - \lb)$
and whose edge set is $E_{\lb} \cup E_{\ub}$, where $E_{\lb} \defeq \set{\supp(\alpha)}[ \alpha \in \Phi_B(\lb, \ub) ]$
and $E_{\ub} \defeq \set{\supp(\beta)}[ \beta \in \Phi_B(\ub, \lb) ]$.
Note that, for each edge $e \in E_{\lb}$ (resp. $e \in E_{\ub}$),
there uniquely exists $\alpha \in \Phi_B(\lb, \ub)$ such that $e = \supp(\alpha)$ (resp. $\beta \in \Phi_B(\ub, \lb)$ such that $e = \supp(\beta)$);
we denote it by $\alpha_e$ (resp. $\beta_e$).
Also note that, by \exaxiom,
for each $u \in \supp( \ub - \lb )$,
there are edges incident to $u$ both in $E_{\lb}$ and in $E_{\ub}$.

For each $e \in E_{\lb}$ and $u \in e$, we choose one edge $f \in E_{\ub}$ with $u \in f$
and construct a triple $(e, u, f)$.
Similarly, for each $f \in E_{\ub}$ and $v \in f$, we choose one edge $e' \in E_{\lb}$ with $v \in e'$
and construct a triple $(f, v, e')$.
Let $T$ be the set of triples constructed as above.
Take any triple $(f_1, u_1, f_2) \in T$ with $f_1 \in E_q$.
By the definition of $T$,
there uniquely exists a triple $(f_2, u_2, f_3) \in T$ such that
$f_2 = \set{u_1, u_2}$.
Similarly, there uniquely exists a triple $(f_3, u_3, f_4) \in T$ such that
$f_3 = \set{u_2, u_3}$.
By repeating the above,
we obtain an infinite walk $(f_1, f_2, f_3, \dots)$ in $G$ such that $f_i \in E_q$ if $i$ is odd and $f_i \in E_r$ if $i$ is even.
By the finiteness of $T$, the infinite walk $(f_1, f_2, f_3, \dots)$ must contain a periodic structure $(\dots, e_1, e_2, \dots, e_{2m}, e_1, \dots)$ with $e_1 \in E_q$.
This implies that there is a closed walk $(e_1, e_2, \dots, e_{2m})$ in $G$ such that $e_i \in E_q$ if $i$ is odd and $e_i \in E_r$ if $i$ is even.

From such a closed walk $(e_1, e_2, \dots, e_{2m})$,
we define
\begin{align*}
    (a_i, \alpha_i) \defeq
    \begin{cases}
    (2, \alpha_{e_{2i-1}}) & \text{if $\card{e_{2i-1}} = 1$},\\
    (1, \alpha_{e_{2i-1}}) & \text{if $\card{e_{2i-1}} = 2$},
    \end{cases}
    \qquad
    (b_i, \beta_i) \defeq
    \begin{cases}
    (2, \beta_{e_{2i}}) & \text{if $\card{e_{2i}} = 1$},\\
    (1, \beta_{e_{2i}}) & \text{if $\card{e_{2i}} = 2$}.
    \end{cases}
\end{align*}
Then, for each $i \in [m]$, we have $\alpha_{e_{2i-1}} \in \Phi_B(\lb, \ub)$ and $\beta_{e_{2i}} \in \Phi_B(\ub, \lb)$.
Moreover, we can see that $\sum_{i = 1}^m a_i \alpha_{e_{2i-1}} + \sum_{i = 1}^m b_i \beta_{e_{2i}} = 0$.
Hence, the lemma holds.
\end{proof}

Let $\alpha_1, \alpha_2, \dots, \alpha_k \in \Phi_B(\lb, \ub)$ and $\beta_1, \beta_2, \dots, \beta_l \in \Phi_B(\ub, \lb)$
satisfying $\sum_{i = 1}^k \alpha_i + \sum_{j = 1}^l \beta_j = 0$;
such ones exist by \Cref{lem:exchange}.
Then, for a sufficiently small $\epsilon > 0$,
we have the following convex combination representation of $p^*$:
\begin{align*}
    p^* = \prn{\lambda^*_q - k \epsilon} q + \prn{\lambda^*_r - l \epsilon} r + \sum_{p \in B \setminus \set{q, r}} \lambda^*_p p + \epsilon \prn{ \sum_{i = 1}^k \prn{ q + \alpha_i } + \sum_{j = 1}^l \prn{ r + \beta_j } }.
\end{align*}
Let $\lambda^\prime = (\lambda^\prime_p)_{p \in B}$ denote the corresponding coefficient of $p^*$,
where we note that $\lambda^\prime \in C$ and $\supp(\lambda^\prime) = \supp(\lambda^*) \cup \set{ q+\alpha_i, r+\beta_j }[i \in [k], j \in [l]]$.

We can compute the value $\theta(\lambda^\prime)$ as follows.
Let
\begin{align*}
    S_+(\lb) &\defeq \set{u \in \supp(\ub - \lb)}[p^*(u) \leq \lb(u) < \ub(u) \text{ or } \ub(u) < \lb(u) \leq p^*(u)],\\
    S_-(\lb) &\defeq \set{u \in \supp(\ub - \lb)}[\lb(u) < \min\set{p^*(u), \ub(u)} \text{ or } \max\set{p^*(u), \ub(u)} < \lb(u)],\\
    s_u(\lb) &\defeq \card{ \set{ i \in [k] }[u \in \supp(\alpha_i)]} \qquad \prn{u \in \supp(\ub - \lb)}.
\end{align*}
We similarly define $S_+(\ub)$, $S_-(\ub)$, and $s_u(\ub)$.
We note that both $\set{S_+(\lb), S_-(\lb)}$ and $\set{S_+(\ub), S_-(\ub)}$ form bipartitions of $\supp(\ub - \lb)$.
Then we have
\begin{align*}
    \theta(\lambda^\prime) = \theta(\lambda^*) + \epsilon \prn{ \sum_{u \in S_+(\lb)} s_u(\lb) - \sum_{u \in S_-(\lb)} s_u(\lb) + \sum_{u \in S_+(\ub)} s_u(\ub) - \sum_{u \in S_-(\ub)} s_u(\ub) }.
\end{align*}
Furthermore, since $\sum_{i = 1}^k \alpha_i + \sum_{j = 1}^l \beta_j = 0$,
we have $s_u(q) = s_u(r)$
for every $u \in \supp(\ub - \lb)$.
Thus, by letting $S_+(\lb, \ub) \defeq S_+(\lb) \cap S_+(\ub)$ and $S_-(\lb, \ub) \defeq S_-(\lb) \cap S_-(\ub)$,
we obtain
\begin{align}\label{eq:decrease}
    \theta(\lambda^\prime) &= \theta(\lambda^*) + \epsilon \sum_{u \in S_+(\lb, \ub)} \prn{s_u(\lb) + s_u(\ub)}  - \epsilon \sum_{u \in S_-(\lb, \ub)} \prn{s_u(\lb) + s_u(\ub)}\notag\\
    &= \theta(\lambda^*) - \epsilon \sum_{u \in S_-(\lb, \ub)} \prn{s_u(\lb) + s_u(\ub)},
\end{align}
where the second equality follows from $S_+(\lb, \ub) = \emptyset$.
By the choice of $\lb$ and $\ub$,
the set $S_-(\lb, \ub)$, which is equal to $\set{u \in \supp(\ub - \lb)}[\lb(u) < p^*(u) < \ub(u) \text{ or } \ub(u) < p^*(u) < \lb(u)]$, contains $u'$; it is nonempty.
By the minimality of $\lambda^*$,
we have $\theta(\lambda^\prime) = \theta(\lambda^*)$,
implying that $\supp(\alpha_i) \cap S_-(\lb, \ub) = \supp(\beta_j) \cap S_-(\lb, \ub) = \emptyset$ for all $i \in [k]$ and $j \in [l]$.

We update $\lambda^*$ and $\ub$ as $\lambda^* \leftarrow \lambda^\prime$ and $\ub \leftarrow \ub + \beta_1$.
For the resulting $\ub$, the value $\norm{\ub - \lb}_1$ strictly decreases,
but the set $S_-(\lb, \ub)$ does not change
since $\supp(\beta_1) \cap S_-(\lb, \ub) = \emptyset$.

By repeating the above update finitely many times,
we finally obtain $\lambda^* \in C$ minimizing $\theta$ and $\lb, \ub \in \supp(\lambda^*)$ such that $S_-(\lb, \ub) = \supp(\ub - \lb)$.
For such $\lb$ and $\ub$,
we again repeat the above update.
Then, since $S_-(\lb, \ub) = \supp(\ub - \lb)$,
the set $S_-(\lb, \ub)$ includes $\supp(\alpha_i)$ and $\supp(\beta_j)$ for any $i \in [k]$ and $j \in [l]$.
Thus, by~\eqref{eq:decrease},
the resulting $\lambda^\prime \in C$ strictly decreases $\theta$,
which contradicts the minimality of $\lambda^*$.

\subsection{Proof of $\text{(a)} \Rightarrow \text{(c)}$}\label{subsec:a=>c}
Suppose that $B = \BP(f) \cap \Z^V$ for an integral bisubmodular function $f$ on $\{-1, 0, 1\}^V$.
For $p \in B$,
let $\Phi_B(p) \defeq \set{\alpha \in \Phi}[p + \alpha \in B]$.
Take any $p, q \in B$.
For $r \in \Z^V$,
we define
\begin{align*}
    v(r) &\defeq \sum \set{ \abs{r(u)} }[u \in V \colon r(u) (q(u)-p(u)) \leq 0 ].
\end{align*}
In particular, the function $v$ for $\alpha \in \Phi_B(p)$ measures the violation of $\alpha$ from $\Phi_B(p, q)$;
for $\alpha \in \Phi_B(p)$, we have that $v(\alpha) = 0$ if and only if $\alpha \in \Phi_B(p, q)$.
The following is an easy observation.
\begin{lemma}\label{lem:v}
Let $r, r^\prime \in \Z^V$.
Then we have
\begin{align*}
    v(r) + v(r^\prime)
    \begin{cases}
        > v(r + r^\prime) & \text{if $r(u) r^\prime(u) < 0$ for some $u \in V$},\\
        = v(r + r^\prime) & \text{otherwise}.
    \end{cases}
\end{align*}
\end{lemma}

We consider the following linear programming:
\begin{align*}
\LP \qquad
\begin{array}{ll}
	\text{Minimize} & \displaystyle\sum \set{v(\alpha) \mu_\alpha}[\alpha \in \Phi_B(p)]\\
	\text{subject to} & \displaystyle\sum\set{\mu_\alpha \alpha}[\alpha \in \Phi_B(p)] = q - p,\\
  &\displaystyle\mu_\alpha \geq 0 \quad (\alpha \in \Phi_B(p)).
	\end{array}    
\end{align*}
It suffices to show that
(i) the problem \LP has a feasible solution,
(ii) its optimal value is $0$,
and (iii) it admits a half-integral optimal solution.
Indeed, statements (i)--(iii)
imply that
$q - p = \sum_{\alpha \in \Phi_B(p, q)} \mu_\alpha \alpha$
for some nonnegative half-integral coefficient $\mu = (\mu_\alpha)_{\alpha \in \Phi_B(p, q)}$,
particularly, that $B$ satisfies \exaxiomm.

Before the proofs of (i) and (ii),
we recall basic facts on $\BP(f)$ and $B$.
For $x \in \set{-1, 0, 1}^V$,
let $\supp^+(x) \defeq \set{u \in V}[x(u) = 1]$
and $\supp^-(x) \defeq \set{u \in V}[x(u) = -1]$.
We define a partial order $\preceq$ on $\set{-1, 0, 1}^V$
by: $x \preceq y$ if and only if $\supp^+(x) \subseteq \supp^+(y)$ and $\supp^-(x) \subseteq \supp^-(y)$.
For $p \in B$ and $u \in V$,
we define
\begin{align*}
    \dep(p, \chi_u) &\defeq \bigsqcap \set{x \in \set{-1,0,1}^V}[u \in \supp^+(x),\ \inpr{p}{x} = f(x)],\\
    \dep(p, -\chi_u) &\defeq \bigsqcap \set{x \in \set{-1,0,1}^V}[u \in \supp^-(x),\ \inpr{p}{x} = f(x)],
\end{align*}
where $\sqcap \emptyset \defeq \zeros$.
By definition,
if $s_u \preceq \dep(p, s_v)$ for some $s_u \in \set{\pm \chi_u}$,
then we have $s_u \preceq \dep(p, s_u) \preceq \dep(p, s_v)$.
The following are known in \cite{MP/AF96}.
\begin{lemma}[\cite{MP/AF96}]\label{lem:basic}
\begin{itemize}
    \item[{\rm (1)}]
    Let $p \in B$, $u \in V$, and $s_u \in \set{\pm \chi_u}$.
    Then we have
    \begin{align*}
        s_u \in \Phi_B(p) \iff p + \epsilon s_u \in \BP(f) \text{ for some } \epsilon > 0 \iff \dep(p, s_u) = \zeros.
    \end{align*}
    \item[{\rm (2)}]
    Let $p \in B$, $u, v \in V$ with $u \neq v$, $s_u \in \set{\pm \chi_u}$, and $s_v \in \set{\pm \chi_u}$.
    Suppose that $p + s_u \notin B$.
    Then we have
    \begin{align*}
        s_u + s_v \in \Phi_B(p) \iff p + \epsilon (s_u + s_v) \in \BP(f) \text{ for some } \epsilon > 0 \iff -s_v \preceq \dep(p, s_u).
    \end{align*}
\end{itemize}
\end{lemma}

We prove statement (i).
By the convexity of $\BP(f)$ and \Cref{lem:cone},
we have $q - p \in \cone(\set{\alpha \in \Phi}[ p + \epsilon \alpha \in \BP(f) \text{ for some } \epsilon > 0 ])$.
Moreover, we can see that
\begin{align*}
    \cone(\set{\alpha \in \Phi}[ p + \epsilon \alpha \in \BP(f) \text{ for some } \epsilon > 0 ]) = \cone(\Phi_B(p)),
\end{align*}
which implies that \LP has a feasible solution, namely, statement (i) holds.
The inclusion ($\supseteq$) is clear.
To see the converse inclusion ($\subseteq$),
take any $\alpha \in \Phi$ such that $p + \epsilon \alpha \in \BP(f)$ for some $\epsilon > 0$.
If $\alpha = s_u$ for some $u \in V$ and $s_u \in \set{\pm \chi_u}$,
then we have $s_u \in \Phi_B(p)$ by \Cref{lem:basic}~(1).
Suppose $\alpha = s_u + s_v$ for some distinct $u,v \in V$, $s_u \in \set{\pm \chi_u}$, and $s_v \in \set{\pm \chi_v}$.
If $s_u, s_v \in \Phi_B(p)$, then $\alpha \in \cone(\Phi_B(p))$; we are done.
Otherwise, by \Cref{lem:basic}~(2),
we have $s_u + s_v \in \Phi_B(p)$.
This completes the proof of (i).

We prove statement (ii).
The following lemma is crucial in the proof.
\begin{lemma}\label{lem:sum}
For $u, v, w \in V$, let $s_u \in \set{\pm \chi_u}$, $s_v \in \set{\pm \chi_v}$, and $s_w \in \set{\pm \chi_w}$.
If $s_u, -s_u + s_v \in \Phi_B(p)$,
then $s_v \in \Phi_B(p)$.
Also if $s_u + s_v, -s_v + s_w \in \Phi_B(p)$ and $s_u \neq -s_w$,
then $s_u + s_w \in \Phi_B(p)$ or $s_u, s_w \in \Phi_B(p)$.
\end{lemma}
\begin{proof}
First we show the former statement.
Note that, by $-s_u + s_v \in \Phi_B(p)$, we have $u \neq v$.
Suppose, to the contrary, that $s_v \notin \Phi_B(p)$.
Then, by $-s_u + s_v \in \Phi_B(p)$ and \Cref{lem:basic}~(2),
we have $s_u \preceq \dep(p, s_u) \preceq \dep(p, s_v)$,
implying $\dep(p, s_u) \neq \zeros$.
Hence $s_u \notin \Phi_B(p)$ by \Cref{lem:basic}~(1),
which contradicts the assumption $s_u \in \Phi_B(p)$.

Next we show the latter statement.
Note that, by $s_u + s_v, -s_v + s_w \in \Phi_B(p)$, we have $u \neq v \neq w$.
Suppose that $s_u \notin \Phi_B(p)$.
By $s_u + s_v \in \Phi_B(p)$ and \Cref{lem:basic}~(2),
we have $-s_v \preceq \dep(p, -s_v) \preceq \dep(p, s_u)$,
implying $\dep(p, -s_v) \neq \zeros$.
Therefore $-s_v \notin \Phi_B(p)$ by \Cref{lem:basic}~(1).
Furthermore, by $-s_v + s_w \in \Phi_B(p)$ and \Cref{lem:basic}~(2),
we obtain $-s_w \preceq \dep(p, -s_w) \preceq \dep(p, -s_v) \preceq \dep(p, s_u)$.
Hence $s_u \neq -s_w$ implies $u \neq w$,
and by \Cref{lem:basic}~(2),
we obtain $s_u + s_w \in \Phi_B(p)$.
\end{proof}

Suppose, to the contrary, that the optimal value of the problem \LP is strictly positive.
Let $\mu^* = (\mu^*_\alpha)_{\alpha \in \Phi_B(p)}$ be an optimal solution.
Take any $\alpha_0 \in \prn{\Phi_B(p) \setminus \Phi_B(p, q)} \cap \supp(\mu^*)$
(such $\alpha_0$ exists by the assumption of the positivity of the optimal value).
Then there is $u \in \supp(\alpha_0)$ such that $\alpha_0(u) (q(u) - p(u)) \leq 0$.
Since $\sum_{\alpha \in \Phi_B(p)} \mu^*_\alpha \alpha = q - p$,
there is $\alpha_1 \in \supp(\mu^*)$ such that $\alpha_0(u) + \alpha_1(u) = 0$, or equivalently, $\alpha_0(u) \alpha_1(u) = -1$.
Then, for $\epsilon \defeq \min \set{ \mu^*_{\alpha_0}, \mu^*_{\alpha_1} }$,
we have the following nonnegative combination representation of $q - p$:
\begin{align}\label{eq:q-p}
    q - p = \prn{\mu^*_{\alpha_0} - \epsilon} \alpha_0 + \prn{\mu^*_{\alpha_1} - \epsilon} \alpha_1 + \sum_{\alpha \in \Phi_B(p) \setminus \set{\alpha_0, \alpha_1}} \mu^*_\alpha \alpha + \epsilon \prn{ \alpha_0 + \alpha_1 }.
\end{align}
It follows from \Cref{lem:sum}
that
$\alpha_0 + \alpha_1 = 0$, $\alpha_0 + \alpha_1 \in \Phi_B(p)$, or there are $s_v \in \set{\pm \chi_v}$ and $s_w \in \set{\pm \chi_w}$ for some $v, w \in V$
such that $s_v, s_w \in \Phi_B(p)$ and $s_v + s_w = \alpha_0 + \alpha_1$.
Hence the RHS of~\eqref{eq:q-p} can be viewed as a nonnegative combination of $\alpha \in \Phi_B(p)$.
Let $\mu^\prime = (\mu^\prime_\alpha)_{\alpha \in \Phi_B(p)}$ denote the corresponding coefficient.

For such $\mu^\prime$,
we have
\begin{align*}
    \sum \set{v(\alpha) \mu^\prime_\alpha}[\alpha \in \Phi_B(p)] &= \sum \set{v(\alpha) \mu^*_\alpha}[\alpha \in \Phi_B(p)] - \epsilon (v(\alpha_0) + v(\alpha_1) - v(\alpha_0 + \alpha_1))\\
    &\leq \sum \set{v(\alpha) \mu^*_\alpha}[\alpha \in \Phi_B(p)] - \epsilon,
\end{align*}
where the inequality follows from $\alpha_0(u) \alpha_1(u) = -1$ and \Cref{lem:v}.
This contradicts the minimality of $\mu^*$.
Therefore, statement (ii) holds.

Statement (iii) immediately follows from the fact~\cite[Example~1 and Theorem~20]{Net/AK06}
that,
for an integer matrix $A = (a_{ij}) \in \Z^{m \times n}$ satisfying $\sum_{i = 1}^n \abs{a_{ij}} \leq 2$ for each column index $j$
and integral vectors $b \in \Z^m$ and $c \in \Z^n$,
the linear programming
$\min \set{\inpr{c}{x}}[Ax = b,\ x \geq 0]$
admits a half-integral optimal solution (if an optimal solution exists).
Since our problem \LP enjoys the above situation and has an optimal solution by (ii),
we obtain statement (iii).

\section*{Acknowledgments}
The author thanks Satoru Fujishige for his valuable comments.
This work was supported by JSPS KAKENHI Grant Numbers JP20K23323, JP20H05795, JP22K17854, Japan.

\end{document}